\documentclass{amsart}

\newtheorem{thm}{Theorem}[section]
\newtheorem{lemma}[thm]{Lemma}
\usepackage{color}
\theoremstyle{definition}
\newtheorem{dn}[thm]{Definition}

\newtheorem{md}[thm]{Proposition}

\theoremstyle{remark}
\newtheorem{remark}[thm]{Remark}

\numberwithin{equation}{section}

\begin{document}
\title{Calder\'on's problem for some classes of conductivities in circularly symmetric domains}
\author[Mai]{Mai Thi Kim Dung}
\address{Department of Mathematics, University of Science-VNU, 334 Nguyen Trai, Thanh Xuan Hanoi}
\email{maithikimdungt59@hus.edu.vn}

\author[Dang]{Dang Anh Tuan}
\address{Department of Mathematics, University of Science-VNU, 334 Nguyen Trai, Thanh Xuan Hanoi}
\email{datuan1105@gmail.com}
\thanks{Part of this work was done when the second author visited the Vietnam Institute for Advanced Study in Mathematics (VIASM)  whom we thank for support and hospitality. We thank N.A.Tu for useful conversation and the referee for helpfull comments.}



\maketitle 
\begin{abstract}
 In this note, we study Calder\'on's problem for certain classes of  conductivities in domains with circular symmetry in two and three dimensions. Explicit formulas are obtained for the reconstruction of the conductivity from the Dirichlet-to-Neumann map. As a consequence, we show that the reconstruction is Lipschitz stable.
\end{abstract}
\section{Introduction}
Consider a conductor in a domain $\Omega\subset \mathbb{R}^n$ with conductivity $\gamma(x).$ When a voltage potential $f\in {H^{\frac{1}{2}}}(\partial \Omega )$ is applied at the boundary $\partial \Omega$, the induced potential $u$ in $\Omega$ is the unique weak solution in $H^1(\Omega)$ of
\begin{align}
\begin{cases}
\nabla  \cdot (\gamma \nabla u)  = 0 & \mbox{ in }\Omega,\\
u = f & \mbox{ on }\partial \Omega.
\end{cases}
\end{align}
The Dirichlet-to-Neumann map is given by $\Lambda _\gamma(f)= \gamma {\partial _\nu }u|_{\partial \Omega}\in {H^{ - \frac{1}{2}}}(\partial \Omega ).$
Here $\nu$ denotes the exterior unit normal to $\partial \Omega.$
The problem studied by Calder\'on in \cite{C} is to determine the conductivity $\gamma$ from  $\Lambda _\gamma$. 

For $n\ge 3$ and $\gamma \in C^2$, that $\Lambda _\gamma$ uniquely determine $\gamma$ was proved by Sylvester and Uhlmann in \cite{SU}. Recently, based on the breakthrough work by Haberman and Tataru \cite{HT}, Caro and Rogers \cite{CR} proved uniqueness for Lipschitz conductivities. There is also related work by Haberman \cite{H}.

In two dimension, and $C^2$ conductivities, the uniqueness was proved by Nachman \cite{N}. Later, Astala and P\"aiv\"arinta   \cite{AP} proved uniqueness for bounded measurable conductivities.

After the uniqueness has been established, it is natural to study the stability of the reconstruction, i.e., we would like to estimate $\gamma_1-\gamma_2$ in certain norm by 
$$||\Lambda_1-\Lambda_2||_\star=\sup\limits_{f\in H^{\frac{1}{2}}(\partial\Omega)\atop f\not=0}\dfrac{||(\Lambda_1-\Lambda_2)f||_{H^{-\frac{1}{2}}(\partial\Omega)}}{||f||_{H^{\frac{1}{2}}(\partial \Omega)}}.$$
In \cite{A0}, Alessandrini proved that the following log-stability estimate hold
$$||\gamma_1-\gamma_2||_{L^\infty(\Omega)}\le C\big(\log(1+||\Lambda_1-\Lambda_2||_\star^{-1})\big)^{-\sigma}$$
where $C, \sigma$ are positive constants and $\gamma_j\in H^{s+2}(\Omega), s>n/2$. Later, Mandache \cite{M} showed that such estimate is optimal.

To improve the stability estimate, Alessandrini and Vessella \cite{A} consider the special classes of piecewise constant conductivities, for $n\ge 3$. The  Lipschitz stability obtained therein has been generalized to other classes of conductivities in \cite{AH}, \cite{BF} and \cite{GS}. 

The analog of the result of \cite{A} was proved for the two dimensional case in \cite{BBR}. Subsequent generalizations of this result are obtained in \cite{BFR} and \cite{Clop}. 

In this paper, we proved Lipschitz stability estimate for two special cases of domains with circular symmetry. 
In the first case, we consider $\Omega=B(0,1)\subset \mathbb{R}^2$  with  conductivities of the form
\begin{align}\label{dung}
{\gamma_\alpha (x)=}\begin{cases}
\alpha_1+\alpha_2(a-r)&  \mbox{ if } 0 \le r<a,  \\
\alpha_0 &\mbox{ if } a\le r <1,
\end{cases}
\end{align}
where $r=|x|$ and $ \varepsilon_0\le \alpha_0, \alpha_1\le M, $ $0\le \alpha_2 \le N.$ We denote this set of conductivities $\mu(a, \varepsilon_0, M,N).$\\
In the second case, we consider $\Omega=B(0,1)\times (0,+\infty)\subset \mathbb{R}^3$ with  conductivities of the form
\begin{align}\label{me}
{\gamma_\alpha (z)=}\begin{cases}
1+{\alpha_1}  &  \mbox{ if }h \le z < \infty,\\
1+ {\alpha_2}   &\mbox{ if } 0 \le z < h ,
\end{cases}
\end{align}
where ${\alpha _j} \in \left[ {0,M} \right],j = 1,2,M > 0,h>0$. We denote this set of conductivities $\mu(h,M).$\\
We give a formula for the Dirichlet-to-Neumann map in each case, together with a formula to recover the conductivity from the Dirichlet-to-Neumann map. As a consequence, we show that the map $\Lambda_{\gamma} \mapsto \gamma$ is Lipschitz.  More precisely our main results are as follows.
\begin{thm}\label{7}
Let $\Omega=B(0, 1)$ and $a \in (0,1), \varepsilon_0, M>0, N\ge 0$. There exists a positive constant $C=C(a, \varepsilon_0,M,N)$ such that
$${\left\| \Lambda _\alpha - \Lambda _\beta \right\|_ \star } \ge C\left( {\left| {{\alpha _0} - {\beta _0}} \right| + \left| {{\alpha _1} - {\beta _1}} \right| + \left| {{\alpha _2} - {\beta _2}} \right|} \right), \forall {\gamma _\alpha },{\gamma _\beta } \in \mu (a,{\varepsilon _0},M,N).$$
\end{thm}
\begin{thm}\label{47}
Let $\Omega=B(0, 1)\times (0, \infty)$ and $h \in (0,\infty), M>0$. There exists a positive constant $C=C(h,M)$ such that
$$\left\| \Lambda _\alpha  - \Lambda _\beta \right\|_{{H_{rad}^{\frac{1}{2}}(B)} \to {H_{rad}^{ - \frac{1}{2}}(B)}} \ge C\left( {\left| {{\alpha _2} - {\beta _2}} \right| + \left| {{\alpha _1} - {\beta _1}} \right|} \right),\forall \gamma _\alpha ,\gamma _\beta\in \mu(h,M).$$
\end{thm}

\section{Proof of Theorem \ref{7}}
Consider the Dirichlet problem  in the unit disc $B=B(0,1)$ on the plane
\begin{align}\label{01}
\begin{cases}
\nabla  \cdot (\gamma_\alpha \nabla u) = 0 & \mbox{ in }B, \\
\hspace{1.55cm}u = f &\mbox{ on }\partial B,
\end{cases}
\end{align}
where the conductivity $\gamma_\alpha\in \mu(a, \varepsilon_0, M, N)$.\\
In the polar coordinate, if $u(x)=\sum\limits_{n\in \mathbb{Z}} {u_n(r)e^{in\theta}} \in H^1(B)$ then the equation in \eqref{01} is 
\begin{equation}\label{02}
\left\{
\begin{aligned}
(\gamma_\alpha u'_{n})' + \frac{\gamma_\alpha }{r}u'_{n} - \frac{{{n^2}\gamma_\alpha }}{r^2}{u_n} &= 0, \forall n\in\mathbb Z,\nonumber \\
\mathop {\lim }\limits_{r \to {a^ - }} {u_n}(r) &= \mathop {\lim }\limits_{r \to {a^ + }} {u_n}(r),  \nonumber \\
\mathop {\lim }\limits_{r \to {a^ - }} \left( {\gamma {u'_n}} \right)(r)& = \mathop {\lim }\limits_{r \to {a^ + }} \left( {\gamma {u'_n}} \right)(r).
\end{aligned}
\right.
\end{equation}
Solving these systems, we obtain
$$u_0(r)=c_0, \hspace{0.5cm} 0\le r<1,$$
and for $n\not=0$,
\begin{align*}
{u_n(r)=}\begin{cases}
{b_n}{r^{\left| n \right|}} + {c_n}{r^{ - \left| n \right|}}&  \mbox{ if }a\le r <1, \\
\sum\limits_{k\ge |n|} {{a_k}{r^k}} ,&\mbox{ if }0< r <a,
\end{cases}
\end{align*}
where
\begin{align*}
{a_{|n| + m}} &= \frac{{{\alpha _2}}}{{{\alpha _1} + a{\alpha _2}}}\frac{{(2m - 1)\left| n \right| + m(m - 1)}}{{2m\left| n \right| + {m^2}}}{a_{\left| n \right| + m - 1}}\\
&={\left( {\frac{{{\alpha _2}}}{{{\alpha _1} + a{\alpha _2}}}} \right)^m}\prod\limits_{j = 1}^m {\frac{{(2j - 1)\left| n \right| + j(j - 1)}}{{2j\left| n \right| + {j^2}}}} {a_{\left| n \right|}}, m=1,2,\cdots.
\end{align*}

\noindent Note that $\alpha_2\ge 0, \alpha_1\ge\epsilon_0>0,$ the power series $\sum_{k\ge |n|}a_kr^k$ is uniformly convergent on $[0, a].$

\noindent From that we get
$$\frac{{{c_n}}}{{{b_n}}} = \frac{{{a^{2\left| n \right|}}\left[ {\left| n \right|{u_n}({a^ - }) - a\frac{{{\alpha _1}}}{{{\alpha _0}}}{u'_n}({a^ - })} \right]}}{{\left| n \right|{u_n}({a^ - }) + a\frac{{{\alpha _1}}}{{{\alpha _0}}}{u'_n}({a^ - })}}.$$
The Dirichlet-to-Neumann map $\Lambda_\alpha:{H^{\frac{1}{2}}}(\partial B) \to {H^{ - \frac{1}{2}}}(\partial B)$ is determined by
$$\Lambda_\alpha f(\theta) =\sum\limits_{n \in \mathbb{Z}} {\widehat {{\Lambda _{{\alpha }}}f}(n){e^{in\theta }}} $$
where $f(\theta)=\sum_{n\in\mathbb Z}\hat{f}(n)e^{in\theta}\in H^{\frac{1}{2}}(\partial B)$ and
\begin{align*}
\widehat {{\Lambda _{{\alpha }}}f}(n) &= \mathop {\lim }\limits_{r \to {1^ - }} {\gamma _\alpha }(r){u'_n}(r) = \alpha_0 \left| n \right|\widehat f(n)\frac{{{b_n} - {c_n}}}{{{b_n} + {c_n}}}\\
&= {\alpha _0}\left| n \right|\widehat f(n)\frac{{1 - {a^{2\left| n \right|}} + (1 + {a^{2\left| n \right|}})\frac{{{\alpha _1}}}{{{\alpha _0}}}{B_n}(b)}}{{1 + {a^{2\left| n \right|}} + (1 - {a^{2\left| n \right|}})\frac{{{\alpha _1}}}{{{\alpha _0}}}{B_n}(b)}},\forall n \in \mathbb{Z},
\end{align*}
$$B_n(b)= 1 + \frac{b}{{2\left| n \right| + 1}}\times \frac{{1 + \sum\limits_{m = 2}^\infty  {m{b^{m - 1}}{h_{m,n}}} }}{{1 + \frac{{\left| n \right|}}{{2\left| n \right| + 1}}b\left( {1 + \sum\limits_{m = 2}^\infty  {{b^{m - 1}}{h_{m,n}}} } \right)}}, $$
$${h_{m,n}} = \prod\limits_{j = 2}^m {\dfrac{{(2j - 1)\left| n \right| + j(j - 1)}}{{2j\left| n \right| + {j^2}}}},  b = \dfrac{{a{\alpha _2}}}{{{\alpha _1} + a{\alpha _2}}}.$$
Note that $0\le b\le b_0=\dfrac{{aN}}{{{\varepsilon _0} + aN}}<1.$\\
To obtain some properties of $B_n(b)$ we need the following technical lemma.
\begin{lemma}\label{6}
(i) $\mathop {\lim }\limits_{n \to \infty } \sum\limits_{m = 2}^\infty  {m{b^{m - 1}}{h_{m,n}}}  = \sum\limits_{m = 2}^\infty  {m{b^{m - 1}}\prod\limits_{j = 2}^m {\dfrac{{2j - 1}}{{2j}}} }={(1 - b)^{ - \frac{3}{2}}} - 1.$\\
(ii) $\mathop {\lim }\limits_{n \to \infty } \sum\limits_{m = 2}^\infty  {{b^{m - 1}}{h_{m,n}}}  = \sum\limits_{m = 2}^\infty  {{b^{m - 1}}\prod\limits_{j = 2}^m {\dfrac{{2j - 1}}{{2j}}} }=\dfrac{2}{{1 - b + \sqrt {1 - b} }} - 1.$
\end{lemma}
We have the following proposition:
\begin{md}\label{5}
$B_n$'s satisfy:
\begin{enumerate}
\item[(i)] $1 \le B_n(b) \le  d_0,$ where $d_0=1 + \frac{{{b_0}}}{{{{(1 - {b_0})}^{\frac{3}{2}}}}}.$
\item[(ii)] $\mathop {\lim }\limits_{n \to \infty } {B_n}(b) = 1.$
\item[(iii)] $\mathop {\lim }\limits_{n \to \infty } (2\left| n \right| + 1)({B_n}(b) - 1) = \dfrac{b}{{1 - b}}.$
\item[(iv)] $\mathop {\lim }\limits_{n \to \infty } \dfrac{{\dfrac{{{\alpha _1}}}{{{\alpha _0}}}{B_n}(b) - 1}}{{\dfrac{{{\alpha _1}}}{{{\alpha _0}}}{B_{n + 1}}(b) - 1}} = 1, b\not=0.$
\item[(v)] $\dfrac{{1 - {b_0}}}{{2\left| n \right| + 1}} \le {B'_n}(b) \le \dfrac{A}{{2\left| n \right| + 1}} \mbox{ where } A=A(a, \varepsilon_0,N) \mbox{ is a constant}.$
\end{enumerate}
\end{md}
\begin{proof}
We rewrite $B_n(b)$ as follows
\begin{equation}\label{4}
{B_n}(b) = 1 + \frac{b}{{2\left| n \right| + 1}}\times \frac{{1 + \sum\limits_{m = 2}^\infty  {m{b^{m - 1}}{h_{m,n}}} }}{{1 + \frac{{\left| n \right|}}{{2\left| n \right| + 1}}b\left( {1 + \sum\limits_{m = 2}^\infty  {{b^{m - 1}}{h_{m,n}}} } \right)}}.
\end{equation}
(i) From \eqref{4} it is easy to see that $B_n(b)\ge 1.$\\
We now show that $B_n(b) \le  d_0.$ Indeed, using (i) in Lemma \ref{5} we have
\begin{align*}
B_n(b)&=\frac{{1 + \frac{{\left| n \right| + 1}}{{2\left| n \right| + 1}}b + \sum\limits_{m = 2}^\infty  {\frac{{\left| n \right| + m}}{{2\left| n \right| + 1}}{b^m}{h_{m,n}}} }}{{1 + \frac{{\left| n \right|}}{{2\left| n \right| + 1}}b + \sum\limits_{m = 2}^\infty  {\frac{{\left| n \right|}}{{2\left| n \right| + 1}}{b^m}{h_{m,n}}} }}\\
&\le 1 + b_0 + \sum\limits_{m = 2}^\infty  {m{b^m_0}}\prod\limits_{j = 2}^m {\dfrac{{2j - 1}}{{2j}}}  = d_0.
\end{align*}
(ii) From \eqref{4} it is not difficult to get $\mathop {\lim }\limits_{n \to \infty } {B_n}(b) = 1$.\\
(iii) We have
 \begin{equation*}
\left( {2\left| n \right| + 1} \right)\left( {{B_n}(b) - 1} \right) = \frac{{b\left( {1 + \sum\limits_{m = 2}^\infty  {m{b^{m - 1}}{h_{m,n}}} } \right)}}{{1 + \frac{{\left| n \right|}}{{2\left| n \right| + 1}}b\left( {1 + \sum\limits_{m = 2}^\infty  {{b^{m - 1}}{h_{m,n}}} } \right)}}.
 \end{equation*}
Hence, from  Lemma \ref{6} we obtain
$$\mathop {\lim }\limits_{n \to \infty } \left( {2\left| n \right| + 1} \right)({B_n}(b) - 1) = \frac{{b{{(1 - b)}^{ - \frac{3}{2}}}}}{{{{(1 - b)}^{ - \frac{1}{2}}}}} = \frac{b}{{1 - b}}.$$
(iv) We consider two cases:\\
$\star$ \underline{\textbf{Case}} $1$: $\alpha_0\ne \alpha_1$. \\
From (ii) we have 
$$\mathop {\lim }\limits_{n \to \infty } \dfrac{{\dfrac{{{\alpha _1}}}{{{\alpha _0}}}{B_n}(b) - 1}}{{\dfrac{{{\alpha _1}}}{{{\alpha _0}}}{B_{n + 1}}(b) - 1}} = \dfrac{{\dfrac{{{\alpha _1}}}{{{\alpha _0}}} - 1}}{{\dfrac{{{\alpha _1}}}{{{\alpha _0}}} - 1}}=1.$$
$\star$ \underline{\textbf{Case}} $2$: $\alpha_0=\alpha_1.$\\
We need to prove $\mathop {\lim }\limits_{n \to \infty } \dfrac{{{B_n}(b) - 1}}{{{B_{n + 1}}\left( b \right) - 1}} = 1$. From (iii) we get
$$\mathop {\lim }\limits_{n \to \infty } \frac{{{B_n}(b) - 1}}{{{B_{n + 1}}(b) - 1}} =\mathop {\lim }\limits_{n \to \infty } \frac{{(2\left| n \right| + 1)\left( {{B_n}(b) - 1} \right)}}{{(2\left| n \right| + 3)\left( {{B_{n + 1}}(b) - 1} \right)}} =  1.$$
(v) We denote by $M_n(b)$ and $N_n(b)$ the numerator and denominator of $B'_n(b)$, respectively. Direct computation gives
\begin{equation}
{M_n}(b) = \frac{1}{{2\left| n \right| + 1}} + \sum\limits_{m = 2}^\infty  {\frac{{\left| n \right|{{(m - 1)}^2}}}{{{{(2\left| n \right| + 1)}^2}}}} {b^m}{h_{m,n}} + \sum\limits_{m = 2}^\infty  {\frac{{{m^2}}}{{2\left| n \right| + 1}}} {b^{m - 1}}{h_{m,n}}+ {I_n}(b), 
\end{equation}
where
$$I_n(b)=  \left( {\sum\limits_{l = 2}^\infty  {\frac{{{l^2}{b^{l - 1}}}}{{2\left| n \right| + 1}}{h_{l,n}}} } \right)\left( {\sum\limits_{k = 2}^\infty  {\frac{{\left| n \right|{b^k}}}{{2\left| n \right| + 1}}{h_{k,n}}} } \right) - \left( {\sum\limits_{l = 2}^\infty  {\frac{{l\left| n \right|{b^{l - 1}}}}{{2\left| n \right| + 1}}{h_{l,n}}} } \right)\left( {\sum\limits_{k = 2}^\infty  {\frac{{k{b^k}}}{{2\left| n \right| + 1}}{h_{k,n}}} } \right).$$
The coefficient of $b^m$ in $I_n(b)$ is:
 \begin{align*}
\sum\limits_{k + l -1= m}  {\left( {\frac{{{l^2}\left| n \right|}}{{{{\left( {2\left| n \right| + 1} \right)}^2}}} - \frac{{l\left| n \right|k}}{{{{\left( {2\left| n \right| + 1} \right)}^2}}}} \right){h_{l,n}}{h_{k,n}}} & = \sum\limits_{k + l-1 = m}  {\frac{{l\left| n \right|(l - k)}}{{{{\left( {2\left| n \right| + 1} \right)}^2}}}{h_{l,n}}{h_{k,n}}}\\
 &=\frac{1}{2}\sum\limits_{k + l-1 = m}  {\frac{{\left| n \right|{{(k - l)}^2}}}{{{{\left( {2\left| n \right| + 1} \right)}^2}}}{h_{l,n}}{h_{k,n}}}.
 \end{align*}
From this we obtain 
\begin{equation}\label{13}
{M_n}(b) \ge \frac{1}{{2\left| n \right| + 1}}.
\end{equation}
Moreover, we have 
\begin{align}\label{9}
\sum\limits_{m = 2}^\infty  {\frac{{\left| n \right|{{(m - 1)}^2}}}{{{{(2\left| n \right| + 1)}^2}}}} {b^m}{h_{m,n}} &\le \frac{1}{{2(2\left| n \right| + 1)}}\sum\limits_{m = 0}^\infty  {{m^2}{b^m}} \nonumber \\
 &= \frac{1}{{2(2\left| n \right| + 1)}} \frac{{b(b + 1)}}{{{{(1 - b)}^3}}} \le \frac{1}{{2(2\left| n \right| + 1)}} \frac{{{b_0}({b_0} + 1)}}{{{{(1 - {b_0})}^3}}}.
\end{align}
Next, we have
\begin{align}\label{11}
\nonumber
\sum\limits_{m = 2}^\infty  {\frac{{{m^2}}}{{2\left| n \right| + 1}}} {b^{m - 1}}{h_{m,n}}& \le \frac{1}{2({2\left| n \right| + 1})}\sum\limits_{m = 0}^\infty  {{m^2}} {b^{m - 1}}\\
 &\le \frac{1}{{2(2\left| n \right| + 1)}} \frac{{{b_0} + 1}}{{{{(1 - {b_0})}^3}}}.
\end{align}
We see that
$$\frac{1}{2}\sum\limits_{k + l - 1 = m} {\frac{{\left| n \right|{{(k - l)}^2}}}{{{{\left( {2\left| n \right| + 1} \right)}^2}}}{h_{l,n}}{h_{k,n}}}  \le \frac{1}{{4(2\left| n \right| + 1)}}{(m + 1)^3},m = 3,4, \ldots $$
It follows that 
\begin{align}\label{12}
\nonumber
{I_n}(b) \le \frac{1}{{4(2\left| n \right| + 1)}}\sum\limits_{m = 3}^\infty  {{{(m + 1)}^3}} {b^m} &\le \frac{1}{{4(2\left| n \right| + 1)}}\sum\limits_{m = 0}^\infty  {{{(m + 1)}^3}} {b^m}\\
& \le \frac{{{b_0}}}{{4\left( {2\left| n \right| + 1} \right)}}\left( {\frac{{2{b_0} + 1}}{{{{(1 - {b_0})}^3}}} + \frac{{3{b_0}({b_0} + 1)}}{{{{(1 - {b_0})}^4}}}} \right).
\end{align}
From \eqref{13}, \eqref{9},\eqref{11} and \eqref{12} 
we deduce that
\begin{equation}\label{14}
\frac{1}{{2\left| n \right| + 1}} \le {M_n}(b) \le \frac{A}{{2\left| n \right| + 1}},\hspace{0.5cm}A\mbox{ is a constant depending on }a,  \varepsilon_0, N.
\end{equation}
On the other hand we have
\begin{equation}\label{15}
1 \le {N_n}(b) = {\left( {1 + \frac{{\left| n \right|}}{{2\left| n \right| + 1}}b + \sum\limits_{m = 2}^\infty  {\frac{{\left| n \right|}}{{2\left| n \right| + 1}}{b^m}{h_{m,n}}} } \right)^2} \le \frac{1}{{1 - b}} \le \frac{1}{{1 - {b_0}}}.
\end{equation}
From \eqref{14} and \eqref{15} we have 
$$\frac{{1 - {b_0}}}{{2\left| n \right| + 1}} \le {B'_n}(b) \le \frac{A}{{2\left| n \right| + 1}},$$
where $A$ is a constant depending on $a, \varepsilon_0, N.$
\end{proof}

We now give an explicit formula to reconstruct the parameters $a$ and $\alpha$ from the Dirichlet-to-Neumann map. We define 
$${C_n}{ = \frac{\Lambda _\alpha ({e^{in\theta }})}{{\left| n \right|{e^{in\theta }}}} = \alpha_0}\frac{{1 - {a^{2\left| n \right|}} + (1 + {a^{2\left| n \right|}})\frac{{{\alpha _1}}}{{{\alpha _0}}}{B_n}(b)}}{1 + {a^{2\left| n \right|}} + (1 - {a^{2\left| n \right|}})\frac{{{\alpha _1}}}{{{\alpha _0}}}{B_n}(b)}.$$
If there is a strictly increasing sequence of positive integers $\{n_k\}_{k=1}^\infty$ such that $C_{n_k}=\alpha_0$, it is easy to obtain $\alpha_0=\alpha_1, \alpha_2=0,$ i.e. the conductor is homogeneous. Otherwise we have the following proposition. 
\begin{md} The following formulas hold
\begin{enumerate}
\item[(i)] $\alpha_0=\mathop {\lim }\limits_{n \to \infty } {C_n}.$
\item[(ii)] $a^{-2}=\mathop {\lim }\limits_{n \to \infty } \frac{{{C_n} - {\alpha _0}}}{{{C_{n + 1}} - {\alpha _0}}}.$
\item[(iii)] ${\alpha _1} = \alpha _0D,$ where
$$D = \frac{{\mathop {\lim }\limits_{n \to \infty } \frac{{{C_n} - {\alpha _0}}}{{2{a^{2\left| n \right|}}{\alpha _0}}} + 1}}{{1 - \mathop {\lim }\limits_{n \to \infty } \frac{{{C_n} - {\alpha _0}}}{{2{a^{2\left| n \right|}}{\alpha _0}}}}}.$$
\item[(iv)] $a\alpha_2=\alpha_1E,$ where
$$ E = \mathop {\lim }\limits_{n \to \infty } (2\left| n \right| + 1)\left[ {\frac{{{\alpha _0}\left( {2{\alpha _0}{a^{2\left| n \right|}} + ({C_n} - {\alpha _0})(1 + {a^{2\left| n \right|}})} \right)}}{{{\alpha _1}\left( {2{\alpha _0}{a^{2\left| n \right|}} - ({C_n} - {\alpha _0})(1 - {a^{2\left| n \right|}})} \right)}} - 1} \right]. $$
\end{enumerate}
\end{md}
\begin{proof}
(i) From (ii) in Proposition \ref{5}
$$\mathop {\lim }\limits_{n \to \infty } {C_n} = {\alpha_0}. $$
(ii) Next we have
$$\mathop {\lim }\limits_{n \to \infty } \frac{{{C_n} - {\alpha _0}}}{{{C_{n + 1}} - {\alpha _0}}} = \mathop {\lim }\limits_{n \to \infty }\frac{{2{a^{2\left| n \right|}}\left( {\frac{{{\alpha _1}}}{{{\alpha _0}}}{B_n}(b) - 1} \right)}}{{2{a^{2\left| n \right| + 2}}\left( {\frac{{{\alpha _1}}}{{{\alpha _0}}}{B_{n + 1}}(b) - 1} \right)}}\frac{{1 + {a^{2\left| n \right| + 2}} + (1 - {a^{2\left| n \right| + 2}})\frac{{{\alpha _1}}}{{{\alpha _0}}}{B_{n + 1}}(b)}}{{1 + {a^{2\left| n \right|}} + (1 - {a^{2\left| n \right|}})\frac{{{\alpha _1}}}{{{\alpha _0}}}{B_n}(b)}}.$$
Using (ii) and (iv) in Proposition \ref{5} we obtain
 $$\mathop {\lim }\limits_{n \to \infty } \frac{{{C_n} - {\alpha _0}}}{{{C_{n + 1}} - {\alpha _0}}} = \frac{1}{{{a^2}}}.$$
(iii) Using (ii) in Proposition \ref{5} we have
\begin{align*}
\mathop {\lim }\limits_{n \to \infty } \frac{{{C_n} - {\alpha _0}}}{{2{a^{2\left| n \right|}}{\alpha _0}}} & = \mathop {\lim }\limits_{n \to \infty }\frac{ {\frac{{{\alpha _1}}}{{{\alpha _0}}}{B_n}(b) - 1} }{{1 + {a^{2\left| n \right|}} + (1 - {a^{2\left| n \right|}})\frac{{{\alpha _1}}}{{{\alpha _0}}}{B_n}(b)}}\\
&=\frac{{\frac{{{\alpha _1}}}{{{\alpha _0}}} - 1}}{{\frac{{{\alpha _1}}}{{{\alpha _0}}} + 1}}.
\end{align*}
This leads to 
$${\alpha _1} = \dfrac{{{\alpha _0}\left( {\mathop {\lim }\limits_{n \to \infty } \dfrac{{{C_n} - {\alpha _0}}}{{2{a^{2\left| n \right|}}{\alpha _0}}} + 1} \right)}}{{1 - \mathop {\lim }\limits_{n \to \infty } \dfrac{{{C_n} - {\alpha _0}}}{{2{a^{2\left| n \right|}}{\alpha _0}}}}}.$$
(iv)  We now calculate $\alpha_2$. 
From 
$${C_n} - {\alpha _0} = {\alpha _0}\frac{{2{a^{2\left| n \right|}}\left( {\frac{{{\alpha _1}}}{{{\alpha _0}}}{B_n}(b) - 1} \right)}}{{1 + {a^2{\left| n \right|}} + (1 - {a^2{\left| n \right|}})\frac{{{\alpha _1}}}{{{\alpha _0}}}{B_n}(b)}}$$
we calculate
$${B_n}(b) = \frac{{{\alpha _0}\left( {2{\alpha _0}{a^{2\left| n \right|}} + ({C_n} - {\alpha _0})(1 + {a^{2\left| n \right|}})} \right)}}{{{\alpha _1}\left( {2{\alpha _0}{a^{2\left| n \right|}} - ({C_n} - {\alpha _0})(1 - {a^{2\left| n \right|}})} \right)}}.$$
From that and (iii) in Proposition \ref{5} we get 
\begin{align*}
E &= \mathop {\lim }\limits_{n \to \infty } (2\left| n \right| + 1)\left[ {\frac{{{\alpha _0}\left( {2{\alpha _0}{a^{2\left| n \right|}} + ({C_n} - {\alpha _0})(1 + {a^{2\left| n \right|}})} \right)}}{{{\alpha _1}\left( {2{\alpha _0}{a^{2\left| n \right|}} - ({C_n} - {\alpha _0})(1 - {a^{2\left| n \right|}})} \right)}} - 1} \right] \\
&= \mathop {\lim }\limits_{n \to \infty } (2\left| n \right| + 1)({B_n}(b) - 1) = \frac{b}{{1 - b}}.
\end{align*}
From $b={\dfrac{{a{\alpha _2}}}{{{\alpha _1} + a{\alpha _2}}}}$ we obtain $a\alpha_2=\alpha_1 E.$
\end{proof}

We now prove Theorem \ref{7}.
\begin{proof}
For $ \gamma_\alpha, \gamma_\beta \in \mu(a,\varepsilon_0,M,N), f\in H^{\frac{1}{2}}(\partial B)$ we have
$$\left\| {\left( \Lambda  _\alpha  - \Lambda _\beta  \right)f} \right\|_{{H^{ - \frac{1}{2}}}(\partial B)}^2 = \sum\limits_{n \in \mathbb{Z}} {\frac{{{n^2}}}{{{{(1 + {n^2})}^{\frac{1}{2}}}}}{{({A_n} - {B_n})}^2}{{\left| {\widehat f(n)} \right|}^2}} ,$$
where $b = a{\alpha _2}/(\alpha _1 + a{\alpha _2}), c = a{\beta _2}/(\beta _1 + a{\beta _2}),$ and
$${A_n} = \alpha_0\frac{{1 - {a^{2\left| n \right|}} + (1 + {a^{2\left| n \right|}})\frac{{{\alpha _1}}}{{{\alpha _0}}}{B_n}(b)}}{{1 + {a^{2\left| n \right|}} + (1 - {a^{2\left| n \right|}})\frac{{{\alpha _1}}}{{{\alpha _0}}}{B_n}(b)}},\hspace{0.5cm}{B_n} = \beta_0\frac{{1 - {a^{2\left| n \right|}} + (1 + {a^{2\left| n \right|}})\frac{{{\beta _1}}}{{{\beta _0}}}{B_n}(c)}}{{1 + {a^{2\left| n \right|}} + (1 - {a^{2\left| n \right|}})\frac{{{\beta _1}}}{{{\beta _0}}}{B_n}(c)}}.$$
By direct computation, we obtain
\begin{align*}
& {A_n} -{B_n} = \frac{{({\alpha _0} + {\beta _0})\left( {\frac{{{\alpha _1}}}{{{\alpha _0}}}{B_n}(b) - \frac{{{\beta _1}}}{{{\beta _0}}}{B_n}(c)} \right)2{a^{2\left| n \right|}}}}{{\left( {1 + {a^{2\left| n \right|}} + (1 - {a^{2\left| n \right|}})\frac{{{\alpha _1}}}{{{\alpha _0}}}{B_n}(b)} \right)\left( {1 + {a^{2\left| n \right|}} + (1 - {a^{2\left| n \right|}})\frac{{{\beta _1}}}{{{\beta _0}}}{B_n}(c)} \right)}}\\
\nonumber
&+\frac{{({\alpha _0} - {\beta _0})\left[ {(1 - {a^{4\left| n \right|}})\left( {1 + \frac{{{\alpha _1}}}{{{\alpha _0}}}\frac{{{\beta _1}}}{{{\beta _0}}}{B_n}(b){B_n}(c)} \right) + (1 + {a^{4\left| n \right|}})\left( {\frac{{{\alpha _1}}}{{{\alpha _0}}}{B_n}(b) + \frac{{{\beta _1}}}{{{\beta _0}}}{B_n}(c)} \right)} \right]}}{{\left( {1 + {a^{2\left| n \right|}} + (1 - {a^{2\left| n \right|}})\frac{{{\alpha _1}}}{{{\alpha _0}}}{B_n}(b)} \right)\left( {1 + {a^{2\left| n \right|}} + (1 - {a^{2\left| n \right|}})\frac{{{\beta _1}}}{{{\beta _0}}}{B_n}(c)} \right)}}.
\end{align*}
We denote by $K_n$ and $H_n$ the numerator and denominator of $A_n-B_n$, respectively. We have ${H_n} \le {\left( {2 + \frac{M}{{{\varepsilon _0}}}{d_0}} \right)^2}$ and 
 $${\left\| \Lambda _\alpha  - \Lambda _\beta  \right\|_{\star}} = \mathop {\sup }\limits_{{f \in {H^{\frac{1}{2}}}(\partial B)}\atop f \ne 0} \frac{{{{\left\| {\left( {{\Lambda _{{\alpha }}} - {\Lambda _{{\beta }}}} \right)f} \right\|}_{{H^{ - \frac{1}{2}}}(\partial B)}}}}{{{{\left\| f \right\|}_{{H^{\frac{1}{2}}}(\partial B)}}}} \ge \mathop {\sup }\limits_{n\not=0} \frac{{\left| {{K_n}} \right|}}{{\left| {{2H_n}} \right|}} \ge \mathop {\sup }\limits_{n\not=0} \frac{{\left| {{K_n}} \right|}}{{{2\left( {2 + \frac{M}{{{\varepsilon _0}}}{d_0}} \right)^2}}},$$
$$\left| {{K_n}} \right| \ge \frac{{2{\varepsilon _0}}}{M}\left| {{\alpha _0} - {\beta _0}} \right| - \frac{{8{M^2}{d_0}}}{{{\varepsilon _0}}}{a^{2\left| n \right|}}.$$
When $\alpha_0 \ne \beta_0$, for $n$ big enough, we obtain 
$$\frac{{8{M^2}{d_0}}}{{{\varepsilon _0}}}{a^{2\left| n \right|}} \le \frac{{{\varepsilon _0}}}{M}\left| {{\alpha _0} - {\beta _0}} \right|.$$
Hence
\begin{equation}\label{h}
{\left\| \Lambda _\alpha  - \Lambda  _\beta  \right\|_ \star } \ge \frac{{{\varepsilon _0}}}{2M}{\left( {2 + \frac{M}{{{\varepsilon _0}}}{d_0}} \right)^{ - 2}}\left| {{\alpha _0} - {\beta _0}} \right|.
\end{equation}
For $\alpha_0 =\beta_0$ we also have \eqref{h}.\\

\noindent Next, we have
$$\left| {{K_n}} \right| \ge 4{\varepsilon _0}{a^{2\left| n \right|}}\left| {\frac{{{\alpha _1}}}{{{\alpha _0}}}{B_n}(b) - \frac{{{\beta _1}}}{{{\beta _0}}}{B_n}(c)} \right| - \left| {{\alpha _0} - {\beta _0}} \right|\left| {1 + \frac{{{M^2}}}{{\varepsilon _0^2}}{d_0} + 4\frac{M}{{{\varepsilon _0}}}} \right|.$$
From  \eqref{h} we have
\begin{equation}\label{19}
{\left\| \Lambda  _\alpha - \Lambda_\beta \right\|_{\star}} \ge {C_1}4{\varepsilon _0}{a^{2\left| n \right|}}\left| {\frac{{{\alpha _1}}}{{{\alpha _0}}}{B_n}(b) - \frac{{{\beta _1}}}{{{\beta _0}}}{B_n}(c)} \right|,
\end{equation}
where $C_1=C_1(a,\varepsilon_0,M)$ is a constant. We now consider
\begin{align*}
\frac{{{\alpha _1}}}{{{\alpha _0}}}{B_n}(b) - \frac{{{\beta _1}}}{{{\beta _0}}}{B_n}(c) =& \frac{{{\alpha _1}}}{{{\alpha _0}}}\left( {{B_n}(b) - {B_n}(c)} \right) + \left( {\frac{{{\alpha _1}}}{{{\alpha _0}}} - \frac{{{\beta _1}}}{{{\beta _0}}}} \right){B_n}(c)\\
=\frac{{{\alpha _1}}}{{{\alpha _0}}}(b - c){B'_n}(\xi ) & +\frac{{{\beta _0}({\alpha _1} - {\beta _1}) + {\beta _1}({\beta _0} - {\alpha _0})}}{{{\alpha _0}{\beta _0}}}{B_n}(c) \hspace{0.5cm}(\text{for } \xi \in (b,c))  \\
= - \frac{{{\beta _1}{B_n}(c)}}{{{\alpha _0}{\beta _0}}}({\alpha _0} - {\beta _0})& + \left[ {\frac{{{B_n}(c)}}{{{\alpha _0}}} - \frac{{{\alpha _1}{\beta _2}a{B'_n}(\xi )}}{{{\alpha _0}({\alpha _1} + {\alpha _2}a)({\beta _1} + {\beta _2}a)}}} \right]({\alpha _1} - {\beta _1})\\
&+\frac{{{\alpha _1}{\beta _1}a{B'_n}(\xi )}}{{{\alpha _0}({\alpha _1} + {\alpha _2}a)({\beta _1} + {\beta _2}a)}}({\alpha _2} - {\beta _2}).
\end{align*}
So from  \eqref{h} and \eqref{19} we have
\begin{equation}\label{23}
 {\left\| \Lambda  _\alpha - \Lambda _\beta \right\|_{\star}} \ge {C_2}{a^{2\left| n \right|}}{D_n},
\end{equation}
where $C_2=C_2(a,\varepsilon_0,M)$ and
$${D_n} = \frac{1}{\alpha_0}\left| \left[ B_n(c) - \frac{{{\alpha _1}{\beta _2}a{B'_n}(\xi )}}{{({\alpha _1} + {\alpha _2}a)({\beta _1} + {\beta _2}a)}} \right]({\alpha _1} - {\beta _1}) + \frac{{{\alpha _1}{\beta _1}a{B'_n}(\xi )}}{{({\alpha _1} + {\alpha _2}a)({\beta _1} + {\beta _2}a)}}({\alpha _2} - {\beta _2}) \right|.$$
Using (i) and (v) in Proposition \ref{5} we get
\begin{equation}\label{16}
\frac{1}{M}\left( {1 - \frac{A}{{2\left| n \right| + 1}}} \right) \le \frac{{{B_n}(c)}}{{{\alpha _0}}} - \frac{{{\alpha _1}{\beta _2}a{B'_n}(\xi )}}{{{\alpha _0}({\alpha _1} + {\alpha _2}a)({\beta _1} + {\beta _2}a)}} \le \frac{{{d_0}}}{{{\varepsilon _0}}},
\end{equation}
\begin{equation}\label{17}
0\le \frac{{\varepsilon _0^2a(1 - {b_0})}}{{M{{({\varepsilon _0} + Na)}^2}\left( {2\left| n \right| + 1} \right)}} \le \frac{{{\alpha _1}{\beta _1}a{B'_n}(\xi )}}{{{\alpha _0}({\alpha _1} + {\alpha _2}a)({\beta _1} + {\beta _2}a)}} \le \frac{A}{{{\varepsilon _0}\left( {2\left| n \right| + 1} \right)}}.
\end{equation}
There exists an $n_0=n_0(a, \varepsilon_0, N)$ such that for every $n\ge n_0$ then
$$0\le \frac{1}{{2M}} \le \frac{1}{M}\left( {1 - \frac{A}{{2\left| n \right| + 1}}} \right).$$
We now show that
\begin{equation}\label{20}
{\left\| \Lambda  _\alpha - \Lambda _\beta \right\|_ \star } \ge C(a,{\varepsilon _0},M,N){\rm{ }}\left( {\left| {{\alpha _1} - {\beta _1}} \right| + \left| {{\alpha _2} - {\beta _2}} \right|} \right).
\end{equation}
We consider three cases.\\
$\star$ \underline{\textbf{Case 1}}: $({\alpha _1} - {\beta _1})({\alpha _2} - {\beta _2}) \ge 0.$\\
We have 
\begin{align}\label{18}
\nonumber
{D_{{n_0}}}& \ge \frac{1}{{2M}}\left| {{\alpha _1} - {\beta _1}} \right| + \frac{{\varepsilon _0^2a(1 - {b_0})}}{{M{{({\varepsilon _0} + Na)}^2}\left( {2\left| {{n_0}} \right| + 1} \right)}}\left| {{\alpha _2} - {\beta _2}} \right|\\
&\ge \min \left\{ {\frac{1}{{2M}},\frac{{\varepsilon _0^2a(1 - {b_0})}}{{M{{({\varepsilon _0} + Na)}^2}\left( {2\left| {{n_0}} \right| + 1} \right)}}} \right\}\left( {\left| {{\alpha _1} - {\beta _1}} \right| + \left| {{\alpha _2} - {\beta _2}} \right|} \right).
\end{align}
From \eqref{23} and \eqref{18} we obtain \eqref{20}.\\
$\star$ \underline{\textbf{Case 2}}: ${\left( {{\alpha _1} - \beta_1 } \right)}\left( {{\alpha _2} - {\beta _2}} \right) < 0$ and
$${D_{{n_0}}} = \left[ {\frac{{{B_{{n_0}}}(c)}}{{{\alpha _0}}} - \frac{{{\alpha _1}{\beta _2}a{B'_n}_{_0}(\xi )}}{{{\alpha _0}({\alpha _1} + {\alpha _2}a)({\beta _1} + {\beta _2}a)}}} \right]\left| {{\alpha _1} - {\beta _1}} \right| - \frac{{{\alpha _1}{\beta _1}a{B'_n}_{_0}(\xi )}}{{{\alpha _0}({\alpha _1} + {\alpha _2}a)({\beta _1} + {\beta _2}a)}}\left| {{\alpha _2} - {\beta _2}} \right|.$$
From that we have
$$\frac{d_0}{{{\varepsilon _0}}}\left| {{\alpha _1} - {\beta _1}} \right| - \frac{{\varepsilon _0^2a(1 - {b_0})}}{{M{{({\varepsilon _0} + Na)}^2}\left( {2\left| {{n_0}} \right| + 1} \right)}}\left| {{\alpha _2} - {\beta _2}} \right| \ge {D_{{n_0}}} \ge 0.$$
Then there exists an $n_1=n_1(a, \varepsilon_0, M,N)> n_0$ such that
$$\frac{{\left| {{\alpha _1} - {\beta _1}} \right|}}{{4M}} \ge \frac{{d_0AM{{({\varepsilon _0} + Na)}^2}}}{{\varepsilon _0^2a(1 - {b_0})}}\frac{{\left( {2\left| {{n_0}} \right| + 1} \right)}}{{\left( {2\left| {{n_1}} \right| + 1} \right)}}\left| {{\alpha _1} - {\beta _1}} \right| \ge \frac{A}{{{\varepsilon _0}\left( {2\left| {{n_1}} \right| + 1} \right)}}\left| {{\alpha _2} - {\beta _2}} \right|.$$
We get
\begin{align}\label{21}
\nonumber
{D_{{n_1}}}& \ge \frac{{\left| {{\alpha _1} - {\beta _1}} \right|}}{{2M}} - \frac{A}{{{\varepsilon _0}\left( {2\left| {{n_1}} \right| + 1} \right)}}\left| {{\alpha _2} - {\beta _2}} \right|\\
& \ge \frac{{\left| {{\alpha _1} - {\beta _1}} \right|}}{{4M}}>\frac{A}{{{\varepsilon _0}\left( {2\left| {{n_1}} \right| + 1} \right)}}\left| {{\alpha _2} - {\beta _2}} \right|.
\end{align}
From \eqref{23} and \eqref{21} we have \eqref{20}\\
$\star$ \underline{\textbf{Case 3}}: ${\left( {{\alpha _1} - \beta_1 } \right)}\left( {{\alpha _2} - {\beta _2}} \right) < 0$ and
$${D_{{n_0}}} = \frac{{{\alpha _1}{\beta _1}a{B'_{{n_0}}}(\xi )}}{{{\alpha _0}({\alpha _1} + {\alpha _2}a)({\beta _1} + {\beta _2}a)}}\left| {{\alpha _2} - {\beta _2}} \right| - \left[ {\frac{{{B_{{n_0}}}(c)}}{{{\alpha _0}}} - \frac{{{\alpha _1}{\beta _2}a{B'_{{n_0}}}(\xi )}}{{{\alpha _0}({\alpha _1} + {\alpha _2}a)({\beta _1} + {\beta _2}a)}}} \right]\left| {{\alpha _1} - {\beta _1}} \right|.$$
There exists an $n_2=n_2(a, \varepsilon_0, M,N)>n_0$ such that 
\begin{equation}\label{25}
\frac{{\varepsilon _0^2a(1 - {b_0})}}{{2M{{({\varepsilon _0} + Na)}^2}\left( {2\left| {{n_0}} \right| + 1} \right)}}\left| {{\alpha _2} - {\beta _2}} \right| \ge \frac{{2d_0MA}}{{\varepsilon _0^2(2\left| {{n_2}} \right| + 1)}}\left| {{\alpha _2} - {\beta _2}} \right|.
\end{equation}
If
$${D_{{n_2}}} = \left[ {\frac{{{B_{{n_2}}}(c)}}{{{\alpha _0}}} - \frac{{{\alpha _1}{\beta _2}aB'_{n_2}(\xi )}}{{{\alpha _0}({\alpha _1} + {\alpha _2}a)({\beta _1} + {\beta _2}a)}}} \right]\left| {{\alpha _1} - {\beta _1}} \right| - \frac{{{\alpha _1}{\beta _1}B'_{n_2}(\xi )}}{{{\alpha _0}({\alpha _1} + {\alpha _2}a)({\beta _1} + {\beta _2}a)}}\left| {{\alpha _2} - {\beta _2}} \right|,$$
we return to \textbf{Case 2}. Otherwise,
\begin{equation}\label{26}
\frac{A}{{{\varepsilon _0}\left( {2\left| {{n_2}} \right| + 1} \right)}}\left| {{\alpha _2} - {\beta _2}} \right| - \frac{1}{{2M}}\left| {{\alpha _1} - {\beta _1}} \right| \ge {D_{{n_2}}} \ge 0.
\end{equation}
From \eqref{25} and \eqref{26} we obtain
$$\frac{{\varepsilon _0^2a(1 - {b_0})}}{{2M{{({\varepsilon _0} + Na)}^2}\left( {2\left| {{n_0}} \right| + 1} \right)}}\left| {{\alpha _2} - {\beta _2}} \right| \ge \frac{d_0}{{{\varepsilon _0}}}\left| {{\alpha _1} - {\beta _1}} \right|.$$
Moreover, we have 
\begin{align}\label{27}
{D_{{n_0}}}& \ge \frac{{\varepsilon _0^2a(1 - {b_0})}}{{M{{({\varepsilon _0} + Na)}^2}\left( {2\left| {{n_0}} \right| + 1} \right)}}\left| {{\alpha _2} - {\beta _2}} \right| - \frac{d_0}{{{\varepsilon _0}}}\left| {{\alpha _1} - {\beta _1}} \right|\notag\\
& \ge \frac{{\varepsilon _0^2a(1 - {b_0})}}{{2M{{({\varepsilon _0} + Na)}^2}\left( {2\left| {{n_0}} \right| + 1} \right)}}\left| {{\alpha _2} - {\beta _2}} \right|\ge \frac{d_0}{{{\varepsilon _0}}}\left| {{\alpha _1} - {\beta _1}} \right|.
\end{align}
From \eqref{23} and \eqref{27} we have \eqref{20}.
From \eqref{h} and \eqref{20} the conclusion follows.
\end{proof}

\section{Proof of Theorem \ref{47}}
Consider the Dirichlet problem
\begin{align}\label{1}
\begin{cases}
\nabla  \cdot (\gamma_\alpha \nabla u) = 0 & \mbox{ in }B\times (0,+\infty), \\
\hspace{1.55cm}u = 0 &\mbox{ on }\partial B \times (0,+\infty),\\
\hspace{1.55cm}u = f &\mbox{ on } B\times \{0\},
\end{cases}
\end{align}
where the conductivity $\gamma_\alpha\in\mu(h, M).$\\
\begin{dn}
(i) We denote $$L_{rad}^2(B) = \left\{ u \in {L^2}(B),u(x, y)=f\left(\sqrt{x^2+y^2}\right) \right\}.$$
$$L_{rad}^2(B\times (0,+\infty)) = \left\{ {u \in {L^2}(B\times (0,+\infty)),} u(x, y, z)=f\left(\sqrt{x^2+y^2}, z\right) \right\}.$$
(ii) Let
$$H_{rad}^{\frac{1}{2}}(B) = \left\{ {f \in L_{rad}^2(B),\sum\limits_{n = 1}^\infty  {{{(1 + {{\left| {{\lambda _n}} \right|}^2})}^{\frac{1}{2}}}{{\left| {\widehat f(n)} \right|}^2}J_1^2(} {\lambda _n}) < \infty } \right\}, $$
where $$\widehat f(n) = \frac{{2\int\limits_0^1 {f(} r){J_0}({\lambda _n}r)rdr}}{{{{({J_1}({\lambda _n}))}^2}}},$$
$J_0(\lambda_n r)$ is Bessel function of order zero, $\lambda_n$ is positive zero of function $J_0$,$${\lambda _1} < {\lambda _2} <  \ldots {\lambda _n} \ldots ,{\lambda _n} \sim \left( {n - \frac{1}{4}} \right)\pi , \mbox{ when }n\to \infty.$$
${J_1}({\lambda _n})$ is Bessel function of order one and
$${J_1}({\lambda _n} ) = \sum\limits_{m = 0}^\infty  {\dfrac{{{{( - 1)}^m}{\lambda _n^{2m + 1}}}}{{{2^{2m + 1}}(m + 1)!m!}}},\hspace{0.5cm}{J_1}({\lambda_n})=-J_0'({\lambda_n})$$
with 
 ${J_1}({\lambda _n}) \sim \sqrt {\dfrac{2}{{\pi {\lambda _n}}}} \cos \left( {{\lambda _n} - \dfrac{{3\pi }}{4}} \right) + {\rm O}\left( {\dfrac{1}{{\lambda _n^{3/2}}}} \right)$ when $n\to\infty.$\\
The norm of $f \in H_{rad}^{\frac{1}{2}}(B)$ is given by
$${\left\| f \right\|_{H_{rad}^{\frac{1}{2}}(B)}} = {\left( {\sum\limits_{n = 1}^\infty  {{{(1 + {{\left| {{\lambda _n}} \right|}^2})}^{\frac{1}{2}}}{{\left| {\widehat f(n)} \right|}^2}J_1^2(} {\lambda _n})} \right)^{\frac{1}{2}}}.$$
(iii) The dual space of $H_{rad}^{\frac{1}{2}}(B) $ is defined by 
 $$H_{rad}^{ - \frac{1}{2}}(B)={\left( {H_{rad}^{\frac{1}{2}}(B)} \right)^*} = \left\{ {f:H_{rad}^{ \frac{1}{2}}(B) \to \mathbb{C}} \mbox{ bounded linear functional }\right\}$$
with norm
$${\left\| f \right\|_{H_{rad}^{ - \frac{1}{2}}(B)}} = {\left( {\sum\limits_{n = 1}^\infty  {{{(1 + {{\left| {{\lambda _n}} \right|}^2})}^{ - \frac{1}{2}}}{{\left| {\widehat f(n)} \right|}^2}J_1^2(} {\lambda _n})} \right)^{\frac{1}{2}}}.$$
(iv) We denote 
$$H_{rad}^1(B \times (0, + \infty )) = \left\{ {u \in L_{rad}^2(B \times (0, + \infty )):|\nabla u| \in L_{rad}^2(B \times (0, + \infty )} \right\}.$$
In the cylindrical coordinates, if $u(r, z)=\sum\limits_{n=1}^\infty u_n(z)J_0(\lambda_n r)$ we have
$${\left\| u \right\|_{{H^1_{rad}}(B \times (0, + \infty ))}} = \pi \sum\limits_{n = 1}^\infty  {J_1^2({\lambda _n})} \int\limits_0^\infty  {{\rm{[}}(1 + \lambda _n^2){{\left| {{u_n}(z)} \right|}^2} + {{\left| {{u'_n}(z)} \right|}^2}} {\rm{]}}dz.$$
\end{dn}
For $f \in {H_{rad}^{ \frac{1}{2}}(B)}$, the Dirichlet problem \eqref{1} in cylindrical coordinates is
\begin{align*}
\begin{cases}
\gamma _\alpha{u_{rr }} + \frac{\gamma _\alpha }{r }{u_r } + \partial _z(\gamma_\alpha{u_z}) = 0,&B \times (0,\infty ), \\
u(1,z)  = 0,& 0<z<\infty,\\
u(r,0)=f,& 0\le r<1,
\end{cases}
\end{align*}
have unique solution $u \in H^1_{rad} \left( {B \times (0,\infty )} \right)$.\\
We expand $u = \sum\limits_{n = 1}^\infty  {{u_n}(z){J_0}({\lambda _n}r )}$. By direct computation we have
\begin{align*}
{u_n(z)=}\begin{cases}
 {a_n}{e^{ - {\lambda _n}z}}&  \mbox{ if }h \le z < \infty,\\
{b_n}{e^{ - {\lambda _n}z}} + {c_n}{e^{{\lambda _n}z}}&\mbox{ if }0 \le z < h.
\end{cases}
\end{align*}
At $z=h$ we have 
\begin{align*}
\begin{cases}
\mathop {\lim }\limits_{z \to {h^ + }} {u_n}(z) = \mathop {\lim }\limits_{z \to {h^ - }} {u_n}(z),\\
\mathop {\lim }\limits_{z \to {h^ + }} \left( \gamma _\alpha {u'_n} \right)(z) = \mathop {\lim }\limits_{z \to {h^ - }} \left( \gamma _\alpha{u'_n} \right)(z).
\end{cases}
\end{align*}
It follows that
\begin{align*}
\frac{c_n}{b_n}=\frac{\alpha_2-\alpha_1}{(2+\alpha_1+\alpha_2)e^{2\lambda_nh}}.
\end{align*}
The Dirichlet-to-Neumann map $\Lambda _\alpha:{H_{rad}^{\frac{1}{2}}(B)} \to {H_{rad}^{ - \frac{1}{2}}(B)}$ is determined by
\begin{equation*}
\Lambda_\alpha f(r) =  - \sum\limits_{n = 1}^\infty  {(1 + {\alpha _2})\frac{{({\alpha _2} - {\alpha _1}){e^{ - 2{\lambda _n}h}} - \left( {2 + {\alpha _1} + {\alpha _2}} \right)}}{{({\alpha _2} - {\alpha _1}){e^{ - 2{\lambda _n}h}} + 2 + {\alpha _1} + {\alpha _2}}}} {\lambda _n}\hat f(n){J_0}({\lambda _n}r ).
\end{equation*}

We now give an explicit formula to reconstruct the parameters $h$, $\alpha$ from the Dirichlet-to-Nemann map. Define
\begin{equation}\label{115}
{A_n} = -\frac{{{\Lambda _{\alpha} }({J_0}({\lambda _n}r))}}{{{\lambda _n}{J_0}({\lambda _n}r)}} = (1 + {\alpha _2})\frac{{2 + {\alpha _1} + {\alpha _2} - ({\alpha _2} - {\alpha _1}){e^{ - 2{\lambda _n}h}}}}{{2 + {\alpha _1} + {\alpha _2} + ({\alpha _2} - {\alpha _1}){e^{ - 2{\lambda _n}h}}}}.
\end{equation}
If $A_1=1+\alpha_2$ then $\alpha_1=\alpha_2,$ i.e. the conductor is homogeneous. Otherwise $A_n\not=1+\alpha_2\; \forall n\in\mathbb N$ and we have the following proposition.
\begin{md} We recontruct $h, \alpha_j$ as follows
\begin{enumerate}
\item[(i)] ${\alpha _2} = \mathop {\lim }\limits_{n \to \infty } {A_n} - 1,$
\item[(ii)] $h = \dfrac{1}{2\pi }\ln \left( {\mathop {\lim }\limits_{n \to \infty } \dfrac{{{A_n} - 1-\alpha_2}}{{{A_{n + 1}} - 1-\alpha_2}}} \right).$
\item[(iii)] ${\alpha _1} = \dfrac{{2A + (A + 2){\alpha _2}}}{{2 - A}},$ where
$$A=\mathop {\lim }\limits_{n \to \infty } \dfrac{{\left( {{A_n} - 1-\alpha_2} \right){e^{2{\lambda _n}h}}}}{{1 + {\alpha _2}}}.$$
\end{enumerate}
\end{md}
\begin{proof}
(i) It is easy to show that ${\alpha _2} = \mathop {\lim }\limits_{n \to \infty } {A_n} - 1.$\\
(ii) We have
$$\frac{{{A_n} - 1-\alpha_2}}{{{A_{n + 1}} - 1-\alpha_2}} = \frac{{{e^{ - 2{\lambda _n}h}}}}{{{e^{ - 2{\lambda _{n + 1}}h}}}}\frac{{2 + {\alpha _1} + {\alpha _2} + ({\alpha _2} - {\alpha _1}){e^{ - 2{\lambda _{n + 1}}h}}}}{{2 + {\alpha _1} + {\alpha _2} + ({\alpha _2} - {\alpha _1}){e^{ - 2{\lambda _n}h}}}}.$$
Note that ${\lambda _n} \sim \left( {n - \frac{1}{4}} \right)\pi , \mbox{ when }n\to \infty$. We obtain
$$\mathop {\lim }\limits_{n \to \infty } \frac{{{A_n} - 1-\alpha_2}}{{{A_{n + 1}} - 1-\alpha_2}} = {e^{2\pi h}}.$$ 
Hence
$$h = \dfrac{1}{2\pi }\ln \left( {\mathop {\lim }\limits_{n \to \infty } \dfrac{{{A_n} - 1-\alpha_2}}{{{A_{n + 1}} - 1-\alpha_2}}} \right).$$
(iii) Since
$$A={\mathop {\lim }\limits_{n \to \infty } \frac{{\left( {{A_n} - 1-\alpha_2} \right){e^{2{\lambda _n}h}}}}{{1 + {\alpha _2}}} = \frac{{2({\alpha _1} - {\alpha _2})}}{{2 + {\alpha _1} + {\alpha _2}}}},$$
so ${\alpha _1} = (2A + (A + 2)\alpha _2)/(2 - A).$
\end{proof}
\begin{remark} We can reconstruct $h, \alpha_1$ from $\alpha_2, A_1, A_2$ as follows
\begin{align*}
h& = \dfrac{1}{2(\lambda_1-\lambda_2)}\ln\left(\frac{(A_1+1+\alpha_2)(A_2-1-\alpha_2)}{(A_1-1-\alpha_1)(A_2+1+\alpha_2)}\right)\\
\alpha_1&= \frac{A_1(2+\alpha_2(1+e^{-2\lambda_1h}))-(1+\alpha_2)(2+\alpha_2(1-e^{-2\lambda_2h}))}{(1+\alpha_2)(1+e^{-2\lambda_1h})-A_1(1-e^{-2\lambda_1h})}.
\end{align*}
\end{remark}

We now prove Theorem \ref{47}.
\begin{proof}
Firstly, for each $\gamma _\alpha, \gamma _\beta\in \mu(h,M), f\in H^{\frac{1}{2}}_{rad}(B)$ we have
$$\left\| {\left( \Lambda _\alpha  - \Lambda _\beta \right)f} \right\|_{H_{rad}^{ - \frac{1}{2}}(B)}^2=\sum\limits_{n = 1}^\infty  {\frac{{\lambda _n^2}}{{{{(1 + \lambda _n^2)}^{ \frac{1}{2}}}}}{\left( A_n - B_n \right)^2}{{|\hat f(n)|}^2}{{\left( {{J_1}({\lambda _n})} \right)}^2}}$$
where
$${A_n} =  -(1 + {\alpha _2}) \frac{{({\alpha _2} - {\alpha _1}){e^{ - 2{\lambda _n}h}} - \left( {2 + {\alpha _1} + {\alpha _2}} \right)}}{{({\alpha _2} - {\alpha _1}){e^{^{ - 2{\lambda _n}h}}} + 2 + {\alpha _1} + {\alpha _2}}},$$
$${B_n} =  - (1 + {\beta _2})\frac{{({\beta _2} - {\beta _1}){e^{ - 2{\lambda _n}h}} - \left( {2 + {\beta _1} + {\beta _2}} \right)}}{{({\beta _2} - {\beta _1}){e^{^{ - 2{\lambda _n}h}}} + 2 + {\beta _1} + {\beta _2}}}.$$
By direct computation we obtain
\begin{equation*}
{A_n} - {B_n} = \frac{{(A - B{e^{ - 2{\lambda _n}h}} - C{e^{ - 4{\lambda _n}h}})({\alpha _2} - {\beta _2}) + D{e^{ - 2{\lambda _n}h}}({\alpha _1} - {\beta _1})}}{{(2 + {\alpha _1} + {\alpha _2} + ({\alpha _2} - {\alpha _1}){e^{ - 2{\lambda _n}h}})(2 + {\beta _1} + {\beta _2} + ({\beta _2} - {\beta _1}){e^{ - 2{\lambda _n}h}})}},
\end{equation*}
where
\begin{align*}
A &= (2 + {\alpha _1} + {\alpha _2})(2 + {\beta _1} + {\beta _2}) \in {\rm{[}}4,4{(M + 1)^2}{\rm{]}},\\
B &= (2 + {\alpha _2} + {\beta _2})(2 + {\beta _1})\in {\rm{[}}4,4{(M + 1)^2}{\rm{]}},\\
C &= ({\alpha _2} - {\alpha _1})({\beta _2} - {\beta _1})\in {\rm{[}} - {M^2},{M^2}{\rm{]}},\\
D &= (2 + {\alpha _2} + {\beta _2})(2 + {\beta _2})\in {\rm{[}}4,4{(M + 1)^2}{\rm{]}}.
\end{align*}
We denote by $K_n$ and $H_n$ the numerator and denominator of $(A_n-B_n),$ respectively. We have $H_n\le (2+3M)^2$ and
\begin{align}\label{111}
\nonumber
{\left\| \Lambda _\alpha - \Lambda_\beta  \right\|_{H_{rad}^{\frac{1}{2}}(B) \to H_{rad}^{ - \frac{1}{2}}(B)}}& = \mathop {\sup }\limits_{f\in H_{rad}^{\frac{1}{2}}(B)\atop f \ne 0} \frac{{{{\left\| {\left( \Lambda _\alpha  - \Lambda _\beta  \right)f} \right\|}_{H_{rad}^{ - \frac{1}{2}}(B)}}}}{{{{\left\| f \right\|}_{H_{rad}^{\frac{1}{2}}(B)}}}}\\
 &\ge \mathop {\sup }\limits_{n\not=0} \frac{{\left| {{K_n}} \right|}}{{\left| {{2H_n}} \right|}} \ge \mathop {\sup }\limits_{n\not=0} \frac{{\left| {{K_n}} \right|}}{{{2\left( {2 + 3M} \right)^2}}},
\end{align}
Hence
$$\mathop {\sup }\limits_n \left| {{K_n}} \right| \ge  (A + Be^{ - 2\lambda_n h} + |C|e^{ - 4\lambda_n h})\left|\alpha _2 - \beta _2 \right| -  D e^{ - 2\lambda_n h}\left|\alpha_1 - \beta_1 \right|.$$
For $\alpha_2\ne \beta_2$, we choose $n$ big enough so that
\begin{equation}\label{112}
 \mathop {\sup }\limits_n  \left| K_n \right| \ge 2\left| {{\alpha _2} - {\beta _2}} \right|.
\end{equation}
From \eqref{112}, \eqref{111} becomes
\begin{equation}\label{40}
{\left\| \Lambda _\alpha - \Lambda_\beta \right\|_{H_{rad}^{\frac{1}{2}}(B) \to H_{rad}^{ - \frac{1}{2}}(B)}} \ge \frac{1}{{{{(2 + 3M)}^2}}}\left| {{\alpha _2} - {\beta _2}} \right|.
\end{equation}
For $\alpha_2=\beta_2$ we also have \eqref{40}.\\
It is easy to get
\begin{equation*}
|K_1|\ge 4e^{-2\lambda_1h}|\alpha_1-\beta_1|-4(M+1)^2(1+e^{-2\lambda_1h})^2|\alpha_2-\beta_2|.
\end{equation*}
Therefore, from \eqref{111} and \eqref{40} we have
\begin{equation}\label{113}
{\left\| \Lambda _\alpha - \Lambda_\beta \right\|_{H_{rad}^{\frac{1}{2}}(B) \to H_{rad}^{ - \frac{1}{2}}(B)}} \ge \frac{e^{-2\lambda_1h}}{{{{2(2 + 3M)}^2}(M+1)^2(1+e^{-2\lambda_1h})^2}}\left| {{\alpha _1} - {\beta _1}} \right|.
\end{equation}
 From \eqref{40} and \eqref{113} we are done.
\end{proof}

\bibliographystyle{amsplain}

\begin{thebibliography}{}
	\bibitem{A0} G. Alessandrini, \textit{Stable determination of conductivity by boundary measurements}, App. Anal., \textbf{ 27} (1988), 153-172.
	
	\bibitem{AH} G. Alessandrini, M. V. de Hoop, R. Gaburro, E. Sincich, \textit{Lipschitz stability for the electrostatic inverse boundary value problem with piecewise linear conductivities},  J. de Math\'ematiques Pures et Appliqu\'ees, \textbf{107}, No. 5 (2017), 638-664.
	
	\bibitem{A} G. Alessandrini, S. Vessella, \textit{Lipschitz stability for the inverse conductivity problem}, Advances in Applied Mathematics, \textbf{ 35} (2005), 207-241.
	
	\bibitem{AP} K. Astala, L. P\"aiv\"arinta, \textit{Calder\'on's inverse conductivity problem in the plane}, Ann. of Math., \textbf{ 163} (2006), 265-299.
	
	
	\bibitem{BBR} J. A. Barcelo, T. Barcelo, A. Ruiz, \textit{Stability of Calder\'on inverse conductivity problem in the plane for less regular conductivities}, J. Differential Equations, \textbf{ 173} (2001), 231-270.
	
	\bibitem{BFR} T. Barcelo, D. Faraco, A. Ruiz, \textit{Stability of Calder\'on inverse conductivity problem in the plane}, J. de Math\'ematiques Pures et Appliqu\'ees, \textbf{ 88}, No. 6 (2007), 552-556.
	
	\bibitem{BF} E. Beretta, E. Francini, \textit{Lipschitz stability for the electrical impedance tomography problem: the complex case}, Comm. in PDEs, \textbf{ 36} (2011), 1723-1749.
	
	\bibitem{C} A. P. Calder\'on, \textit{On an inverse boundary value problem}, Seminar on Numerical Analysis and its Applications to Continuum Physics (Rio de Janeiro, 1980),65-73, Soc. Brasil Mat., Rio de Janeiro, 1980, Reprinted in: Comput. Appl. Math. \textbf{ 25}, No. 2-3 (2006), 133-138.
	
		
	\bibitem{CR}
	P. Caro, K. Rogers, \emph{Global uniqueness for the Calder\'on problem with Lipschitz conductivities}, Forum Math. Pi 4 (2016), e2, 28 pp. 
	
	\bibitem{Clop} 
	A. Clop, D. Faraco, A. Ruiz, \textit{Stability of Calderon' s inverse conductivity problem in the plane for discontinuous conductivities}, Inverse Problems and Imaging,\textbf{Vol 4, No. 1} (2010), 49-91.

	
	\bibitem{GS} R. Gaburro, E. Sincich, \textit{Lipschitz stability for the inverse conductivity problem for a conformal class of anisotropic conductivities}, Inverse Problems, \textbf{ 31} 015008 (2015).
	
\bibitem{H} B. Haberman, \textit{Uniqueness in Calder\'on problem for conductivities with unbounded gradient}, Comm. Math. Phys., \textbf{340}, No. 2 (2015), 639-659.
	
	\bibitem{HT} B. Haberman, D. Tataru \textit{Uniqueness in Calder\'on problem with Lipschitz conductivities}, Duke Math. Journal. \textbf{162}, No. 3 (2013), 497-516.				
	
	
	\bibitem{M} N. Mandache, \textit{Exponential instability in an inverse problem for Schrodinger equation}, Inverse Problems, \textbf{ 17} (2001), 1435-1444.
	
	\bibitem{N} A. Nachman, \textit{Global uniqueness for a two-dimensional inverse boundary value problem }, Ann. of Math., \textbf{ 142} (1995), 71-96.
	
			
	\bibitem{SU} J. Sylvester, G. Uhlmann, \textit{A global uniqueness theorem for an inverse boundary value problem}, Ann. of Math., \textbf{ 125} (1987), 153-169.
\end{thebibliography}

\end{document}